\documentclass{amsart}
  
\usepackage{amsmath}
\usepackage{amsthm}
\usepackage{amssymb}
\usepackage{mathrsfs}
\usepackage{booktabs}
\usepackage{siunitx}
\usepackage{multirow}
\usepackage{cite}
\usepackage[left=2.95cm,right=2.95cm, bottom=3cm]{geometry}
\usepackage{array}
\usepackage{enumerate}
\usepackage{color}
\usepackage[colorlinks=true, linkcolor=red, citecolor=blue]{hyperref}
\numberwithin{equation}{section}
\usepackage{enumitem}
\setlist[enumerate,1]{label={\upshape(\roman*)}}

\theoremstyle{theorem}
\newtheorem{theorem}{Theorem}[section]

\newtheorem{proposition}[theorem]{Proposition}

\newtheorem{lemma}[theorem]{Lemma}

\theoremstyle{definition}

\newtheorem{remark}[theorem]{Remark}

\theoremstyle{remark}
\theoremstyle{proof}

\newcommand{\R}{\mathbb{R}}

\newcommand{\E}{\mathbb{E}}
\renewcommand{\P}{\mathbb{P}}

\newcommand{\cG}{\mathcal{G}}
\newcommand{\cH}{\mathcal{H}}
\newcommand{\cL}{\mathcal{L}}

\newcommand{\sfA}{\mathsf{A}}
\newcommand{\sfB}{\mathsf{B}}
\newcommand{\sfC}{\mathsf{C}}

\newcommand{\wt}{\widetilde}

\newcommand{\varep}{\varepsilon}

\newcommand{\dive}{\operatorname{div}}

\begin{document}

\title{Quantitative inequalities for the expected lifetime of Brownian motion} 
\author{Daesung Kim}
\address{Department of Mathematics, Purdue University \\ 150 N. University Street, West Lafayette, IN 47907-2067, USA}
\email{daesungkim@purdue.edu}
\begin{abstract}
The isoperimetric inequalities for the expected lifetime of Brownian motion state that the $L^p$-norms of the expected lifetime in a bounded domain for $1\leq p\leq \infty$ are maximized when the region is a ball with the same volume. In this paper, we prove quantitative improvements of the inequalities.  Since the isoperimetric properties hold for a wide class of L\'evy processes, many questions arise from these improvements. 
\end{abstract}
\date{\today}

\maketitle

\section{Introduction}
Finding stability estimates has been of current interest in the study of functional and geometric inequalities such as the Sobolev inequalities \cite{Bianchi1991a, Cianchi2009a, Chen2013a}, the Hardy--Littlewood--Sobolev inequality \cite{Carlen2017a}, the logarithmic Sobolev inequality \cite{Indrei2014a,Fathi2016a, Indrei2018a,Feo2017a, Dolbeault2016a, Kim2018a}, the Hausdorff--Young inequality \cite{Christ2014a}, the isoperimetric inequalities \cite{Fusco2008a, Figalli2010a, Fusco2011a}, and the Faber--Krahn inequalities \cite{Brasco2015a, Brasco2019a}. Generally speaking, a functional or geometric inequality can be written as 
\begin{equation*}
    \cG(u)\geq c\cH(u)
\end{equation*}
where $\cG$ and $\cH$ are nonnegative functionals on a class of admissible functions or sets. The inequality is called \emph{sharp} if the constant $c$ cannot be replaced by any larger number. It is called \emph{optimal} if there exists $u_0$ such that $\cG(u_0)=c\cH(u_0)$. Such $u_0$ is called an \emph{optimizer}. For an optimal inequality, the deficit is defined by $\delta(u)=\cG(u)-c\cH(u)\geq 0$. Once the class of optimizers are characterized, a natural question is to measure the deviation of $u$ from optimizers when $\delta(u)$ gets close to 0. In particular, a lower bound of $\delta(u)$ in terms of a distance of $u$ from the class of optimizers is called a \emph{stability estimate} or a \emph{quantitative improvement}.

Let $\alpha\in(0,2]$ and $D$ a bounded domain in $\R^{n}$. Let $X^{\alpha}_t$ be the rotationally symmetric $\alpha$--stable process with generator $-(-\Delta)^{\alpha/2}$. Let $\tau^{\alpha}_{D}$ be the first exit time of $X^{\alpha}_{t}$ from $D$ and $u^{\alpha}_{D}(x)=\E^{x}[\tau^{\alpha}_{D}]$ the expected lifetime where $\E^{x}$ is the expectation associated with $X^\alpha_t$ starting at $x\in\R^n$. Note that $u^{\alpha}_{D}(x)$ is a solution to the equation
\begin{equation*}
    \begin{cases}
        (-\Delta)^{\frac{\alpha}{2}}u(x)=1, & x\in D, \\
        u(x)=0, & x\notin D.
    \end{cases}
\end{equation*}
If $B$ is a ball of radius $R$ and centered at the origin, then $u^{\alpha}_{B}(x)$ is explicitly given by
\begin{equation*}
    u^{\alpha}_{B}(x)=C_{n,\alpha}(R^{2}-|x|^{2})^{\frac{\alpha}{2}}.
\end{equation*}
For $\alpha=2$, $X_t^\alpha$ is Brownian motion with generator $\Delta$. In this case, we drop the superscript $\alpha$. 

Ba\~nuelos and M\'endez-Hern\'andez \cite{Banuelos2010a} showed that several isoperimetric type inequalities for Brownian motion continue to hold for a wide class of L\'evy processes using the symmetrization of L\'evy processes and the multiple integral rearrangement inequalities of Brascamp--Lieb--Luttinger \cite{Brascamp1974a}. Indeed, they proved in \cite[Theorem 1.4]{Banuelos2010a} that if a L\'evy process $Y_t$ has an absolutely continuous L\'evy measure with respect to the Lebesgue measure, and if $f$ and $V$ are nonnegative continuous functions, then for any $x\in D$ and $t>0$,
\begin{equation*}
    \E^0[f^\ast(Y^\ast_t)\exp\left(\int_0^t V^\ast(Y^\ast_s)\, ds\right);\tau_B^{Y^\ast_t}>t]
    \geq \E^x[f(Y_t)\exp\left(\int_0^t V(Y_s)\, ds\right);\tau_D^{Y_t}>t]
\end{equation*}
where $f^\ast$ and $V^\ast$ are the symmetric decreasing rearrangements of $f$ and $V$, $Y_t^\ast$ is the symmetrization of $Y_t$ defined in \cite[p.4029]{Banuelos2010a}, and $B$ is a ball centered at 0 with $|D|=|B|$. A particular case  of this is that for all $t\geq 0$ and $x\in\R^n$, 
\begin{equation}\label{eq:stableDist}
    \P^{0}(\tau_{B}^{\alpha}>t)\geq \P^{x}(\tau_{D}^{\alpha}>t),
\end{equation}
which yields 
\begin{equation}\label{eq:stableELI}
    u^{\alpha}_{B}(0)\geq u^{\alpha}_{D}(x),
\end{equation}
where $B$ is a ball centered at 0 with $|B|=|D|$. In fact,  \eqref{eq:stableDist}  gives 
\begin{equation}\label{eq:stable-pmmt}
    \E^0(\tau_{B}^{\alpha})^p\geq \E^x(\tau_{D}^{\alpha})^p
\end{equation}
for all $p>0$. Talenti \cite{Talenti1976b} proved that the $L^p$ norm of a solution of a second-order elliptic equation is maximized when the elliptic operator and the domain are symmetrically rearranged. In particular, the result yields that for $p>0$, $\alpha=2$, and a bounded domain $D$,
\begin{equation}\label{eq:Talenti}
    \|u_{B}\|_p\geq \|u_D\|_p
\end{equation}
where $B$ is a ball with $|B|=|D|$. 

Given the above isoperimetric type inequalities for the first exit time of the $\alpha$-stable process and their connection to the classical torsion function, there are many questions that arise concerning quantitative versions of these inequalities. The goal of this paper is to study quantitative versions of the expected lifetime inequalities \eqref{eq:stableELI} for $\alpha=2$ and \eqref{eq:Talenti} for $p\geq 1$. We define the deficit of \eqref{eq:stableELI} by
\begin{equation}\label{eq:eetineq}
    \delta(x,D)=1- \frac{u_{D}(x)}{u_{B}(0)}\geq 0
\end{equation}
where $B$ is a ball centered at 0 with $|B|=|D|$. The first main result is a lower bound of the deficit $\delta(x,D)$ in terms of the deviations of $x$ and $D$ from the class of optimizers. Note that equality holds in \eqref{eq:eetineq} if $D$ is a ball and $u_{D}(x)=\max_{y\in D}u_D(y)$. The deviation of $x$  is represented by the level set $|\{y\in D:u_D(y)>u_D(x)\}|$, and the deviation of $D$ by the Fraenkel asymmetry, which is defined by
\begin{equation*}
    A(D)=\inf\left\{\frac{|D\triangle B|}{|D|}: B\text{ is a ball with }|B|=|D|\right\}. 
\end{equation*}

\begin{theorem}\label{mainthm1}
Let $D\subseteq \R^{n}$ be a bounded domain with $A(D)>0$. Let $D_{t}=\{y\in D: u_{D}(y)>t\}$, $\mu(t)=|D_{t}|$, and
\begin{equation}\label{eq:tstar}
    t_{\ast}=t_{\ast}(D)=\sup\left\{t>0:\mu(t)>|D|(1-\frac{1}{4}A(D))\right\}.
\end{equation}
Then we have
\begin{equation}\label{eq:main1}
    \delta(x,D)        
    \geq |D|^{-\frac{2}{n}}\Big(\mu(u_{D}(x))^{\frac{2}{n}}+\sfC_{n}(u_{D}(x)\wedge t_{\ast})A(D)^{2}\Big),
\end{equation}
where $\sfC_n=\beta_{n}\omega_{n}^{\frac{1}{n}}$, $\beta_n$ is a dimensional constant in \eqref{eq:quantiso}, and $\omega_n$ is the volume of a unit ball in $\R^n$.
\end{theorem}

The proof is based on the proof of \eqref{eq:stableELI} for $\alpha=2$ in \cite{Banuelos1987a, Talenti1976b}, and the sharp quantitative isoperimetric inequality \cite{Fusco2008a}. In order to estimate the asymmetry of the level sets, we use the idea of Hansen and Nadirashvili \cite{Hansen1994a} as in the proof of the boosted P\'olya--Szeg\"o inequality \cite[Lemma 2.9]{Brasco2016a}.  

\begin{remark}\label{rmk1}
We note that \eqref{eq:main1} with the first remainder term follows from the pointwise estimate $u_B(x)\geq (u_D)^\ast(x)$ of \cite{Talenti1976a} where $(u_D)^\ast(x)$ is the symmetric decreasing rearrangement of $u_D(x)$ and $B$ is a ball centered at 0 with $|D|=|B|$. For simplicity, we assume $|D|=1$.    For each $x\in D$, we define $r:D\to [0,\infty)$ by $\mu(u_D(x))=|B_{r(x)}|$ where $B_{r(x)}$ is a ball of radius $r(x)$. For a nonnegative measurable function $f$ on $D$, the symmetric decreasing rearrangement $f^\ast(x)=f^\ast(|x|)$ satisfies $f^\ast(r(x))\geq f(x)$ for each $x\in D$. Since $u_B$ is rotationally symmetric, we use the notation $u_B(x)=u_B(|x|)$. By $u_B(x)\geq (u_D)^\ast(x)$, one has 
\begin{align*}
    u_D(x)
    \leq (u_D)^\ast(r(x))
    \leq u_B(r(x))
    =u_B(0)(1-(\omega_n^{\frac{1}{n}}r(x))^2)
    =u_B(0)(1-\mu(u_D(x))^{\frac{2}{n}}).
\end{align*} 
Notice that \eqref{eq:main1} can be written as $u_B(r(x))-u_D(x)\geq \sfC_n(u_D(x)\wedge t_\ast)A(D)^2$.
\end{remark}

\begin{remark}
Note that if $A(D)>0$, then $t_\ast>0$. Suppose $\delta(x,D)=0$. If $A(D)>0$, then \eqref{eq:main1} implies $\mu(u_D(x))=0$ and $u_D(x)=0$, which is a contradiction. Thus $D$ is a ball with $|B|=|D|$. As a consequence, one sees that equality holds \eqref{eq:eetineq} \emph{only if} $D$ is a ball and $u_{D}(x)=\max_{y\in D}u_D(y)$.
\end{remark}

\begin{remark}
One can extend the result to a wide class of second-order elliptic operators as in \cite{Talenti1976b}. Let $\cL=\partial_i(a_{ij}(x)\partial_j)$ where $a_{ij}(x)$ is a bounded measurable function with
\begin{align}\label{eq:elliptic}
    \sum_{i,j=1}^{n}a_{ij}(x)\xi_i\xi_j\geq \sum_{i=1}^n \xi_i^2
\end{align}
for each $x\in\R^n$ and $\xi=(\xi_1,\cdots,\xi_n)\in\R^n$. Consider a weak solution $u^{\cL}_D$ of
\begin{align*}
    \begin{cases}
        -\cL u(x)=1, & x\in D,\\
        u(x)=0, & x\in\partial D.
    \end{cases}
\end{align*}
Following the proof of Theorem \ref{mainthm1} and modifying \eqref{eq:deri_energy} with inequality, which follows from the elliptic condition \eqref{eq:elliptic}, one obtains
\begin{align*}
    1-\frac{u^\cL_D(x)}{u_B(0)}
    \geq |D|^{-\frac{2}{n}}\Big(\mu(u^\cL_{D}(x))^{\frac{2}{n}}+C_{n}(u^\cL_{D}(x)\wedge t_{\ast})A(D)^{2}\Big).
\end{align*} 
\end{remark}

The second result is a quantitative inequality for the $L^p$ norm of the expected lifetime, $1\leq p\leq\infty$. We define the $L^p$ deficit of the expected lifetime inequality \eqref{eq:Talenti} for $1\leq p\leq\infty$ by
\begin{align*}
    \delta_p(D)=
    1-\left(\frac{\|u_D\|_p}{\|u_{B}\|_p}\right)^{\kappa(p)}
\end{align*}
where $\kappa(p)=p$ for $1\leq p<\infty$, $\kappa(\infty)=1$, and $B$ is a ball with $|B|=|D|$.

\begin{theorem}\label{mainthm2}
Let $n\geq 2$ and $D$ be a bounded domain in $\R^n$. For $1\leq p\leq \infty$, we have
\begin{equation}\label{eq:quantLpelt}
    \delta_p(D)
    \geq \sfC_{n,p}A(D)^{2+\kappa(p)}
\end{equation}
where $\sfC_{n,p}$ is explicitly given in \eqref{eq:Cnp_1} and \eqref{eq:Cnp_2}.
\end{theorem}

\begin{remark}
The torsional rigidity of $D$ is defined by $T(D)=\|u_D\|_1$. The Saint-Venant inequality states that the torsional rigidity is maximized when the region is a ball with the same volume. If $p=1$, Theorem \ref{mainthm2} produces the non-sharp quantitative Saint-Venant inequality
\begin{align}\label{eq:quantSV}
    T(B)-T(D)\geq \sfC_{n,1}T(B)A(D)^{3},
\end{align}
which was proven in \cite{Brasco2016a}. Thus Theorem \ref{mainthm2} can be thought of as an extension of \eqref{eq:quantSV} to the case $1<p\leq \infty$.
\end{remark}

\begin{remark}\label{rmk:ellipse}
It is natural to ask if the exponent of $A(D)$ in \eqref{eq:quantLpelt} is sharp. Let $n=2$ and $\varep>0$. Consider an ellipse $D=\{(x,y)\in\R^2:x=\cos t, y=(1+\varep)\sin t, t\in\R\}$. The asymmetry of $D$ is $A(D)=\frac{1}{\pi}\varep+O(\varep^2)$ (see \cite[pp. 88--89]{Hall1991a}).
Note that the torsion function of $D$ is 
\begin{align*}
    u_D(x)=\frac{(1+\varep)^2}{2(1+(1+\varep)^2)}\left(1-x^{2}-\frac{y^2}{(1+\varep)^2}\right).
\end{align*}
Let $B$ be a ball with $|B|=|D|=(1+\varep)\pi$. Let $p\in[1,\infty)$. A direct computation shows that
\begin{align*}
    \|u_B\|_p^p-\|u_D\|_p^p
    &=\frac{\pi}{2^{2p}(p+1)}(1+\varep)^{p+1}-\frac{\pi}{2^p(p+1)(1+(1+\varep)^2)^p}(1+\varep)^{2p+1}\\
    &=\frac{\pi}{2^{2p}(p+1)}(1+\varep)^{p+1}\left(
        1-\Big(1-\frac{\varep^2}{1+(1+\varep)^2}\Big)^p
    \right)\\
    &=C_p\varep^2+o(\varep^2)
\end{align*}
for some $C_p>0$, and 
\begin{align*}
    \delta_\infty(D)
    =1-\frac{\|u_D\|_\infty}{\|u_B\|_\infty}
    =1-\frac{2(1+\varep)}{1+(1+\varep)^2}
    =\frac{\varep^2}{1+(1+\varep)^2}.
\end{align*}
This observation shows that the exponent of $A(D)$ in \eqref{eq:quantLpelt} cannot be replaced by any smaller number than 2. It is expected that the sharp exponent would be 2, which is an interesting open problem. 
\end{remark}

Brasco, De Philippis, and Velichkov \cite{Brasco2015a} showed that the sharp exponent of \eqref{eq:quantSV} is 2 in a sense that the power cannot be replaced by any smaller number. Their method, however, does not give an explicit dimensional constant because the proof relies on the selection principle of Cicalese and Leonardi \cite{Cicalese2012a}. 

The key step in the proof of Theorem \ref{mainthm2} is the removal of $t_\ast$ defined in \eqref{eq:tstar}. In \cite{Brasco2016a}, the authors proved the non-sharp quantitative Saint-Venant inequality \eqref{eq:quantSV} using transfer of asymmetry (\ref{lem:prop_asym}) and the boosted P\'olya--Szeg\"o inequality. Thus $t_\ast$ also appears in their proof. To replace $t_\ast$ by $A(D)$ (up to a dimensional constant), they made use of the variational representation for $T(D)$ \eqref{eq:TD-varfor}. In our case, however, the $L^p$ norm of the expected lifetime does not have an appropriate variational formula for $1<p\leq \infty$. Instead, we estimate the distribution function of $u_D$ when $t_\ast$ is sufficiently small and apply the layer cake representation and the strong Markov property. It turns out that this enables us to replace $t_\ast$ by $A(D)$. 

The fractional analogue of \eqref{eq:quantSV} is proven in \cite{Brasco2019a}. They show that if $n\geq 2$, $\alpha\in(0,2)$, and $D$ is an open set with $|D|=1$, then 
\begin{align*}
    T_\alpha(B)-T_\alpha(D)\geq C_{n,\alpha}A(D)^{\frac{6}{\alpha}}
\end{align*}
where $C_{n,\alpha}$ is explicit and $B$ is a ball with $|B|=1$. Here $T_\alpha(D)$ is the fractional torsional rigidity defined in \eqref{eq:Talpha-def}. Furthermore, they prove that if $D$ has Lipschitz boundary and satisfies the exterior ball condition, then the exponent can be lowered to $2+\frac{2}{\alpha}$. It turns out that our method for removing $t_\ast$ yields the same exponent without any additional geometric assumptions on $D$.

\begin{theorem}\label{main3}
    If $n\geq 2$, $\alpha\in(0,2)$, and $D$ is an open set with $|D|=1$, then 
    \begin{align*}
        T_\alpha(B)-T_\alpha(D)\geq C_{n,\alpha}A(D)^{2+\frac{2}{\alpha}}
    \end{align*}
    where $B$ is a ball with $|B|=|D|$.
\end{theorem}

The paper is organized as follows. In Section \ref{sec:prlm}, we review basic facts about the torsional rigidity and transfer of asymmetry, which is a key idea of the main results. Indeed, we give a proof of Lemma \ref{lem:prop_asym}, which was essentially proven by Hansen and Nadirashvili \cite{Hansen1994a}, and explain how this idea can be applied to our context (see Remark \ref{rmk:transferasymmetry}). In Section \ref{sec:pfs}, we give proofs of the main results.  In Section \ref{sec:prbls}, we discuss related open problems; extensions of Theorem \ref{mainthm1} and Theorem \ref{mainthm2} to the $\alpha$-stable processes, finding the sharp exponent of the main results, and quantitative improvements of \eqref{eq:stableDist} and \eqref{eq:stable-pmmt} even for Brownian motion. We also investigate the extension of Caffarelli--Silvestre \cite{Caffarelli2007a}, a fractional P\'olya--Szeg\"o inequality with a remainder term, and its relation to the fractional torsional rigidity. 

\section{Preliminaries}\label{sec:prlm}

\subsection{Torsional rigidity}

Let $\alpha\in(0,2]$ and $D$ a bounded domain in $\R^{n}$. Let $X^{\alpha}_t$ be the rotationally symmetric $\alpha$--stable process with generator $-(-\Delta)^{\alpha/2}$. Let $\tau^{\alpha}_{D}$ be the first exit time of $X^{\alpha}_{t}$ from $D$ and $u^{\alpha}_{D}(x)=\E^{x}[\tau^{\alpha}_{D}]$ the expected lifetime where $\E^{x}$ is the expectation associated with $X^\alpha_t$ starting at $x\in\R^n$. Let $P_{t}^\alpha$ be the semigroup associated with $X^\alpha_t$ killed upon exiting $D$ given by $P^\alpha_tf(x)=\E^x[f(X^\alpha_t); t<\tau_D^\alpha]$ on $L^2(D)$. The general semigroup theory yields (see \cite{Davies1989a}) that there exists an orthonormal basis $\{\varphi_{n}\}$ of $L^2(D)$ and the corresponding eigenvalues $0<\lambda_1<\lambda_2\leq \lambda_3\leq\cdots$ such that $P^\alpha_t\varphi_n=e^{-t\lambda_n}\varphi_n$ and $(-\Delta)^{\alpha/2}\varphi_n=\lambda_n\varphi_n$. Using the representation of the transition density of $X^\alpha_t$
\begin{equation*}
    p_t(x,y)=\sum_{n=1}^{\infty}e^{-\lambda_n t}\varphi_n(x)\varphi_n(y),
\end{equation*}
one obtains
\begin{equation*}
    \P^x(\tau_D^\alpha>t)
    =\int_D p_t(x,y)dy=\sum_{n=1}^{\infty}e^{-\lambda_n t}\|\varphi_n\|_1\varphi_n(x)
\end{equation*}
and 
\begin{equation*}
    u_D^\alpha(x)
    =\int_0^\infty\P^x(\tau_D^\alpha>t)dt 
    =\sum_{n=1}^{\infty}\frac{\|\varphi_n\|_1}{\lambda_n}\varphi_n(x).
\end{equation*}
In addition (see \cite[Theorem 4.4]{Bogdan2009a}), there exist constants $c_1, c_2$ depending on $D$ and $\alpha$ such that $c_1 u_D^\alpha(x)\leq \varphi_1(x)\leq c_2 u_D^\alpha(x)$ for all $x\in D$. For further information, we refer the reader to \cite{Bogdan2009a} and the references therein. 

The classical torsional rigidity of $D$ is defined by $T(D)=\|u_D\|_1$ for $\alpha=2$. In this context, $u_D(x)$ is also called the torsion function of $D$. Let $W^{1,2}_{0}(D)$ be the completion of $C^{\infty}_{0}(D)$ with respect to the norm $u\mapsto \|\nabla u\|_{2}$. We have variational representations of the torsional rigidity
\begin{equation}\label{eq:TD-varfor}
    T(D)
    =\max
    \left\{\frac{\|u\|_{1}^{2}}{\|\nabla u\|_{2}^{2}}:u\in W^{1,2}_{0}(D),u\neq 0\right\}
    =\max
    \left\{2\|u\|_{1}-\|\nabla u\|_{2}^{2}:{u\in W^{1,2}_{0}(D)}\right\}. 
\end{equation}
Since $u_{D}$ is an optimizer for \eqref{eq:TD-varfor}, we have $T(D)=\|u_D\|_1=\|\nabla u_D\|_2^2$. There are two important inequalities concerning the torsional rigidity. The Saint-Venant inequality, which is an isoperimetric inequality for $T(D)$, states that if $D$ is a set of finite measure in $\R^n$ then 
\begin{equation}\label{eq:saint-venant}
    |B|^{-\frac{n+2}{n}}T(B)\geq |D|^{-\frac{n+2}{n}}T(D)
\end{equation}
where $B$ is a ball. The second is the Kohler-Jobin inequality, which states that 
\begin{equation*}
    \lambda_1(D)T(D)^{\frac{2}{n+2}}\geq \lambda_1(B)T(B)^{\frac{2}{n+2}}.
\end{equation*}
Note that the classical Faber--Krahn inequality for the first eigenvalue $\lambda_1$ follows from these two inequalities for $T(D)$. Indeed, one has
\begin{equation*}
    \frac{\lambda_1(D)}{\lambda_1(B)}
    \geq \left(\frac{T(B)}{T(D)}\right)^{\frac{2}{n+2}}
    \geq \left(\frac{|B|}{|D|}\right)^{\frac{2}{n}}.
\end{equation*} 
Furthermore, the authors in \cite{Brasco2015a} showed that stability of the Saint-Venant inequality implies that of the Faber--Krahn inequality for the first semilinear eigenvalue via the Faber--Krahn hierarchy.  

The fractional torsional rigidity for $0<\alpha<2$ is defined by 
\begin{equation}\label{eq:Talpha-def}
    T_\alpha(D)
    =\int_D u_D^\alpha(x)\, dx
    =\int_D\int_0^{\infty} \P^{x}(\tau_{D}^{\alpha}>t) \,dt dx.
\end{equation}
There has been recent progress in the study of the fractional torsional rigidity. The isoperimetric inequality for $T_\alpha(D)$, a fractional analogue of the Saint-Venant inequality, follows from \cite[Corollary 5.4]{Banuelos2010a} where the isoperimetric inequality was proven for a general class of L\'evy processes. For the stable processes, it also follows from the sharp rearrangement inequality of \cite[Theorem A.1]{Frank2008b}. Recently, Brasco, Cinti, and Vita \cite{Brasco2019a} derived a quantitative version of the fractional Faber--Krahn inequality and that of the fractional Saint-Venant inequality as a corollary. Their method is based on the extension of \cite{Caffarelli2007a} and the symmetrization argument of \cite{Fusco2011a}. 

\subsection{Transfer of asymmetry}
The following lemma is essentially from \cite[Lemma 5.1]{Hansen1994a}, which provides an estimate of asymmetries of two sets when these sets are close in $L^1$ sense. We refer the reader to \cite[Lemma 4.1]{Brasco2019a} for its generalization. 
\begin{lemma}[{\cite[Lemma 2.8]{Brasco2016a}}]\label{lem:prop_asym}
    Let $D\subseteq \R^{n}$ be an open set with finite measure, $U\subseteq D$, $|U|>0$, and
    \begin{equation*}
        \frac{|D\setminus U|}{|D|}\leq kA(D)
    \end{equation*}
    for $k\in(0,\frac{1}{2})$.
    Then, $A(U)\geq (1-2k)A(D)$.
\end{lemma}
\begin{proof}
Let $B_{1}$ be a ball centered at 0 with $|B_{1}|=|U|$ satisfying 
\begin{equation*}
    A(U)=\frac{|U\triangle (x+B_{1})|}{|U|}
\end{equation*}
for some $x\in\R^{n}$ and $B_{2}$ a ball centered at 0 with $|B_{2}|=|D|$. Note that $|U\triangle D|=|D\setminus U|=|B_{1}\triangle B_{2}|$. Using the triangular inequality for the symmetric difference, one can see that
\begin{align*}
    A(U)
    &= \frac{|U\triangle (x+B_{1})|}{|U|}\\
    &\geq \frac{|D\triangle (x+B_{2})|-|U\triangle D|-|B_{1}\triangle B_{2}|}{|D|}\\
    &\geq A(D)-2\frac{|D\setminus U|}{|D|}\\
    &\geq (1-2k)A(D).
\end{align*}
\end{proof}

\begin{remark}\label{rmk:transferasymmetry}
Let $D$ be a bounded domain in $\R^n$, $u$ a nonnegative function defined in $D$, and $D_{t}=\{x:u(x)>t\}$ for $t>0$. Assume $A(D)>0$ and 
\begin{equation*}
    t_{\ast}=\sup\{t>0:\mu(t)>|D|(1-\frac{1}{4}A(D))\}>0.
\end{equation*}
If $t<t_{\ast}$, then we have
\begin{equation}\label{eq:levelset_est}
    \frac{|D\setminus D_t|}{|D|}
    =1-\frac{\mu(t)}{|D|}
    \leq 1-(1-\frac{1}{4}A(D))
    =\frac{1}{4}A(D),
\end{equation}
which yields $A(D_t)\geq \frac{1}{2}A(D)$ by Lemma \ref{lem:prop_asym}.     
\end{remark}

\section{Proofs of the main results}\label{sec:pfs}

\subsection{Proof of Theorem \ref{mainthm1}}

Suppose $|D|=1$. Let $D_{t}=\{x\in D:u(x)>t\}$, $\mu(t)=|D_{t}|$, and $u(x)=u_{D}(x)$. By the coarea formula, we have
\begin{align}\label{eq:coarea}
    \left(-\frac{d}{dt}\int_{D_t}|\nabla u|\, dx\right)^2
    \geq P(D_t)^2
\end{align}
for almost every $t>0$. Note that the sharp quantitative isoperimetric inequality \cite{Fusco2008a} states
\begin{equation}\label{eq:quantiso}
    P(D)\geq P(B)+\beta_n A(D)^2
\end{equation}
where $B$ is a ball with $|B|=|D|=1$ and $\beta_n$ is a dimensional constant. A simple manipulation gives
\begin{align}\label{eq:quantisoperi}
    P(D_{t})^{2}
    &\geq P(D_{t}^{\ast})^{2}+2P(D_{t}^{\ast})(P(D_{t})-P(D_{t}^{\ast}))\\
    &\geq P(D_{t}^{\ast})^{2}+(2n\omega_{n}^{\frac{1}{n}}\beta_{n})\mu(t)^{2-\frac{2}{n}}A(D_{t})^{2}\nonumber\\
    &\geq n^2\omega_{n}^{\frac{2}{n}}\mu(t)^{2-\frac{2}{n}}\left(1+\frac{2}{n}\beta_{n}\omega_{n}^{-\frac{1}{n}}A(D_{t})^{2}\right)\nonumber
\end{align}
where $\omega_{n}$ is the volume of the unit ball in $\R^{n}$ and $D^\ast_t$ is a ball with $|D_t|=|D^\ast_t|$. It follows from Cauchy--Schwarz inequality that 
\begin{align}\label{eq:csineq}
    (-\mu'(t))^{\frac{1}{2}}
    \left(-\frac{d}{dt}\int_{D_t}|\nabla u|^2\, dx\right)^{\frac{1}{2}}
    \geq -\frac{d}{dt}\int_{D_t}|\nabla u|\, dx. 
\end{align}
Combining \eqref{eq:coarea}, \eqref{eq:quantisoperi}, and \eqref{eq:csineq}, we get
\begin{align*}
    -\mu'(t)
    \left(-\frac{d}{dt}\int_{D_t}|\nabla u|^2\, dx\right)
    \geq n^2\omega_{n}^{\frac{2}{n}}\mu(t)^{2-\frac{2}{n}}\left(1+\frac{2}{n}\beta_{n}\omega_{n}^{-\frac{1}{n}}A(D_{t})^{2}\right)
\end{align*}
for almost every $t>0$. Since $u$ is a weak solution of $-\Delta u=1$ in $D$, 
\begin{align*}
    \int_D \varphi\, dx = \int_D \nabla u\cdot \nabla \varphi \,dx
\end{align*}
for all $\varphi\in W^{1,2}_0(D)$. Let $\varphi(x)=(u(x)-t)_+$, then it belongs to $\varphi\in W^{1,2}_0(D)$ and
\begin{align*}
    \int_{D_t}(u-t)\, dx = \int_{D_t}|\nabla u|^2\, dx.
\end{align*}
For small $h\in\R$, 
\begin{align*}
    \frac{1}{h}\left(
        \int_{D_t}|\nabla u|^2\, dx-\int_{D_{t+h}}|\nabla u|^2\, dx
        \right) 
    =\mu(t+h)+\int_{D_t\triangle D_{t+h}}\Big|\frac{u-t}{h}\Big|\, dx.
\end{align*}
Since $0\leq |u-t|\leq |h|$ in $D_t\triangle D_{t+h}$ and $|D_t\triangle D_{t+h}|\to 0$ as $h\to 0$, we obtain
\begin{align}\label{eq:deri_energy}
    \mu(t)= -\frac{d}{dt}\int_{D_t}|\nabla u|^2\, dx.
\end{align}
Therefore, we have
\begin{equation}\label{eq:thm_3_1}
    -\mu(t)^{\frac{2}{n}-1}\mu'(t)
    \geq n^{2}\omega_{n}^{\frac{2}{n}}\left(1+\frac{2}{n}\beta_{n}\omega_{n}^{-\frac{1}{n}}A(D_{t})^{2}\right)
\end{equation}
for almost every $t>0$.

For each $t>0$, choose $R(t)>0$ such that $\mu(t)=|B_{R(t)}(0)|$, where $B_{R(t)}(0)$ is the ball of radius $R(t)$, centered at 0. Let $\tau_{R(t)}$ be the first exit time from $B_{R(t)}(0)$. Since $\E^{x}[\tau_{R(t)}]=\frac{1}{2n}(R(t)^{2}-|x|^{2})$, we have
\begin{equation}\label{eq:eltformula}
    \E^{0}[\tau_{R(t)}]
    =\frac{1}{2n}\omega_{n}^{-\frac{2}{n}}\mu(t)^{\frac{2}{n}}.
\end{equation}
Differentiating of the both sides in $t$ and applying \eqref{eq:thm_3_1}, we have
\begin{equation*}
    -\frac{d}{dt}\E^{0}[\tau_{R(t)}]
    =-\frac{1}{n^{2}}\omega_{n}^{-\frac{2}{n}}\mu(t)^{\frac{2}{n}-1}\mu'(t)
    \geq 1+\frac{2}{n}\beta_{n}\omega_{n}^{-\frac{1}{n}}A(D_{t})^{2}
\end{equation*}
for almost every $t>0$. Taking the integral over $[0, u_D(x)]$ and applying \eqref{eq:eltformula}, we have
\begin{align*}
    u_B(0)-\frac{1}{2n}\omega_{n}^{-\frac{2}{n}}\mu(u_{D}(x))^{\frac{2}{n}} 
    &=\E^{0}[\tau_{R(0)}]-\E^{0}[\tau_{R(u_D(x))}]\\
    &\geq u_D(x)+\frac{2}{n}\beta_{n}\omega_{n}^{-\frac{1}{n}}\int_{0}^{u(x)}A(D_{t})^{2}\,dt.
\end{align*}
By Lemma \ref{lem:prop_asym} and Remark \ref{rmk:transferasymmetry}, we have $A(D_{t})\geq \frac{1}{2}A(D)$ for $t<t_{\ast}$ and
\begin{equation*}
    \int_{0}^{u(x)}A(D_{t})^{2}\,dt
    \geq \int_{0}^{u(x)\wedge t_{\ast}}A(D_{t})^{2}\,dt
    \geq \frac{1}{4}(u(x)\wedge t_{\ast})A(D)^{2}.
\end{equation*}
Therefore, we obtain
\begin{align*}
    u_B(0)-u_D(x)
    &\geq \frac{1}{2n\omega_{n}^{n/2}}\mu(u_{D}(x))^{\frac{2}{n}} +\frac{2}{n}\beta_{n}\omega_{n}^{-\frac{1}{n}}\int_{0}^{u(x)}A(D_{t})^{2}\,dt\\
    &\geq u_B(0)\left(\mu(u_{D}(x))^{\frac{2}{n}} +\sfC_{n}(u(x)\wedge t_{\ast})A(D)^{2}\right)
\end{align*}
where $\sfC_n=\beta_{n}\omega_{n}^{\frac{1}{n}}$.

Suppose that $|D|=r^{-n}$ for some $r>0$. By translation invariance, we assume $0\in D$ without loss of generality. For $r>0$, we denote by $rD=\{ry:y\in D\}$. Note that the Fraenkel asymmetry is scaling invariant, i.e. $A(D)=A(rD)$. By the scaling property of  $X_t$, we have $r^{2}u_D(x)=u_{rD}(rx)$. This leads to the following scaling identities
    \begin{align*}
        \delta(x,D)&= \delta(rx,rD),\\ 
        \mu_D(t)&= |\{y:u_D(y)>t\}|
            =|\{y:u_{rD}(ry)>r^2 t\}|
            =r^{-n}\mu_{rD}(r^2 t),\\
        t_\ast(D)
            &=\sup\{t>0:\mu_D(t)>|D|(1-\frac{1}{4}A(D))\}\\
            &=\sup\{t>0:\mu_{rD}(r^2 t)>|rD|(1-\frac{1}{4}A(rD))\} \\
            &=r^{-2}t_\ast(rD).
    \end{align*}
Since $|rD|=1$, we have
\begin{align*}
    \delta(x,D)
    &= \delta(rx, rD) \\
    &\geq  \mu(u_{rD}(rx))^{\frac{2}{n}} +\sfC_{n}(u_{rD}(rx)\wedge t_{\ast}(rD))A(rD)^{2}\\
    &= r^2\left(\mu(u_{D}(x))^{\frac{2}{n}} +\sfC_{n}(u(x)\wedge t_{\ast})A(D)^{2}\right)\\
    &= |D|^{-\frac{2}{n}}\left(\mu(u_{D}(x))^{\frac{2}{n}} +\sfC_{n}(u(x)\wedge t_{\ast})A(D)^{2}\right),
\end{align*}
as desired.
\qed

\subsection{Proof of Theorem \ref{mainthm2}}
If $A(D)=0$, the results follow from \eqref{eq:Talenti}. From now on, we assume $A(D)>0$. By scaling invariance, we assume $|D|=1$ without loss of generality. Let $B$ be a ball centered at 0 with $|B|=1$.

Consider $p\in[1,\infty)$. Let $D_{t}=\{x\in D:u_D(x)>t\}$ and $\mu_D(t)=\mu(t)=|D_{t}|$. Note that Theorem \ref{mainthm1} reads
\begin{align*}
    \frac{1}{2n\omega_n^{2/n}}(1-\mu(u_D(x))^{2/n})-u_D(x)\geq \tilde{\sfC}_n (u_D(x)\wedge t_{\ast})A(D)^2
\end{align*}
where $\tilde{\sfC}_n=\frac{1}{2n\omega_n^{2/n}}\sfC_n$. By the coarea formula, we have
\begin{align*}
    \frac{1}{(2n)^p \omega_n^{2p/n}}\int_D (1-\mu(u_D(x))^{2/n})^p \,dx
    &=\frac{1}{(2n)^p \omega_n^{2p/n}} \int_0^\infty\int_{\partial D_{t}}(1-\mu(u_D(x))^{2/n})^p|\nabla u_D|^{-1}d\cH^{n-1}(x)\,dt\\
    &= -\frac{1}{(2n)^p \omega_n^{2p/n}}\int_0^\infty (1-\mu(t)^{2/n})^p\mu'(t)\,dt\\
    &= \frac{1}{2^{p+1}n^{p-1}\omega_n^{2p/n}}B(p,(n-2)/2)\\
    &=\|u_B\|_p^p
\end{align*}
where $B(a,b)$ is the Beta function. Using $a^p-b^p\geq p b^{p-1}(a-b)$ for $a\geq b$, we get
\begin{align}\label{eq:destwitht}
    \|u_B\|_p^p-\|u_D\|_p^p
    &\geq \tilde{\sfC}_n A(D)^2\int_D p u_D(x)^{p-1}(u_D(x)\wedge t_{\ast})\,dx \\
    &\geq \tilde{\sfC}_nA(D)^2\int_0^{t_{\ast}} pt^{p-1}\mu(t)\, dt\nonumber\\
    &\geq \frac{1}{2}\tilde{\sfC}_n  A(D)^2 \, (t_{\ast})^{p}.\nonumber
\end{align}
In the last inequality, we used the fact that $\mu(t)>|D|(1-\frac{1}{4}A(D))\geq\frac{1}{2}$ for $0<t<t_\ast$.

Let $\mu_B(t)=|\{x\in B: u_B(x)>t\}|$. Since $u_B(x)=\frac{1}{2n}(r_n^2-|x|^2)$ with $r_n=\omega_n^{-\frac{1}{n}}$, we have
\begin{align}\label{eq:mu0t}
    \mu_B(t)=\big(1-2n\omega_n^{\frac{2}{n}}t\big)^{\frac{n}{2}}.
\end{align}
Choose $t_0>0$ so that $\mu_B(2t_0)=1-\frac{1}{8}A(D)$. By \eqref{eq:mu0t} and the inequality $1-(1-x)^a\geq ax $ for $0\leq x, a\leq 1$, we have
\begin{align}\label{eq:t0est}
    t_0
    =\frac{1}{4n\omega_n^{\frac{2}{n}}}\left(1-(1-\frac{1}{8}A(D))^{\frac{2}{n}}\right)
    \geq  \frac{1}{16n^2\omega_n^{\frac{2}{n}}}A(D).
\end{align}
Suppose $t_\ast<t_0$, then $\mu_D(t)\leq 1-\frac{1}{4}A(D)$ for all $t> t_0$ by definition. Since $\mu_B(t)\geq 1-\frac{1}{8}A(D)$ for $t\leq 2t_0$, we get $\mu_B(t)-\mu_D(t)\geq \frac{1}{8}A(D)$ for $t\in(t_0,2t_0]$. By the layer cake representation and \eqref{eq:t0est}, we have
\begin{align*}
    \|u_B\|_p^p-\|u_D\|_p^p
    &= \int_0^\infty pt^{p-1}(\mu_B(t)-\mu_D(t))\, dt \\
    &\geq \int_{t_0}^{2t_0} pt^{p-1}(\mu_B(t)-\mu_D(t))\, dt \\
    &\geq \frac{p}{8} (t_0)^{p}A(D)\\
    &\geq \frac{p}{2^{4p+3}n^{2p}\omega_n^{\frac{2p}{n}} } A(D)^{1+p}\\ 
    &\geq \frac{p}{2^{4(p+1)}n^{2p}\omega_n^{\frac{2p}{n}} } A(D)^{2+p}.
\end{align*}
If $t_\ast\geq t_0$, then it follows from \eqref{eq:destwitht} and \eqref{eq:t0est} that 
\begin{align*}
    \|u_B\|_p^p-\|u_D\|_p^p
    \geq 
    \frac{\tilde{\sfC}_n}{2^{4p+1}n^{2p}\omega_n^{\frac{2p}{n}}}  A(D)^{2+p}. 
\end{align*}
For $1\leq p<\infty$, we finish the proof of \eqref{eq:quantLpelt} by letting
\begin{align}\label{eq:Cnp_1}
    \sfC_{n, p}
    &=\frac{1}{2^{4(p+1)}n^{2p}\omega_n^{\frac{2p}{n}} \|u_B\|_p^p} \min\{p, 8\tilde{\sfC}_n\}\\
    &=\frac{1}{2^{3(p+1)}n^{p+1} B(p,(n-2)/2)} \min\left\{p, \frac{4\beta_{n}}{n\omega_n^{\frac{1}{n}}}\right\}\nonumber
\end{align}
where $\beta_{n}$ is the constant in \eqref{eq:quantiso}.

Consider the case $p=\infty$. By translation invariance, we assume that $0\in D$ and $u_D(0)=\max_{y\in D}u_D(y)$ without loss of generality. Putting $x=0$ in \eqref{eq:main1}, we get 
\begin{align*}
    \delta_\infty(D)
    \geq \sfC_n  t_\ast A(D)^2.
\end{align*} 
Let $\mu_B(t)=|\{x\in B: u_B(x)>t\}|$ and choose $t_0>0$ so that $\mu_B(2t_0)=1-\frac{1}{8}A(D)$ as above. If $t_\ast\geq t_0$, then it follows from \eqref{eq:t0est} that
\begin{align*}
    \delta_\infty(D)
    \geq  \frac{\sfC_n}{16n^2\omega_n^{\frac{2}{n}}}  A(D)^3.
\end{align*}
Let $t_\ast<t_0$. Let $\varep>0$ be small enough that $t_1:=t_\ast+\varep<t_0$ and $D_1=\{x\in D: u_D(x)>t_1\}$, then $D_1$ is open. Let $\tilde{B}$ be a ball centered at 0 with $|\tilde{B}|=|D_1|$ and $\tilde{t}$ be such that $\mu_B(\tilde{t})=\mu_D(t_1)$. Since $1-\frac{1}{4}A(D)>\mu_D(\tilde{t})$, we have $\tilde{t}>2t_0$. Recall that the strong Markov property of $X_t$ yields for any $x\in U\subset D$ that
\begin{align*}
    \E^x[\tau_D]=\E^x[\tau_{U}]+\E^x[\E^{X_{\tau_U}}[\tau_{D}]].
\end{align*}
Since the paths of $X_t$ are continuous a.s., we have $X_{\tau_{D_1}}\in\partial D_1$ a.s. Since $D_1$ is open, $\partial D_1\subset \R^n\setminus D_1$ and $u_D(y)\leq t_1$ for $y\in \partial D_1$. Then we obtain
\begin{align*}
    \E^0[\tau_D]=\E^0[\tau_{D_1}]+\E^0[\E^{X_{\tau_{D_1}}}[\tau_{D}]]\leq \E^0[\tau_{D_1}]+t_1.
\end{align*}
On the other hand, it follows from a direct computation that $\E^0[\tau_B]=\E^0[\tau_{\tilde{B}}]+\tilde{t}$. Since $\E^0[\tau_{\tilde{B}}]\geq \E^0[\tau_{D_1}]$ by \eqref{eq:stableELI}, we get 
\begin{align*}
    \|u_B\|_\infty-\|u_D\|_\infty
    &= u_B(0)-u_D(0)\\
    &\geq (\E^0[\tau_{\tilde{B}}]-\E^0[\tau_{D_1}])+t_0\\
    &\geq \frac{1}{16n^2\omega_n^{\frac{2}{n}}}A(D)\\
    &\geq \frac{\|u_B\|_\infty}{32n}A(D)^3.
\end{align*}
We complete the proof by letting 
\begin{align}\label{eq:Cnp_2}
    \sfC_{n,\infty}=\frac{1}{32n^2}\min\left\{\ 2\beta_n \omega_n^{-\frac{1}{n}}, n \right\}.
\end{align}
\qed

\subsection{Proof of Theorem \ref{main3} }

Since $A(D)<2$, it suffices to consider the case $\frac{1}{2}T_\alpha(B) \leq T_\alpha(D)$. Let $u_D^\alpha$ be the expected lifetime of the $\alpha$-stable process in $D$, $\mu_D(t)=\mu(t)=|\{y\in D: u_D^\alpha(y)>t\}|$, and $t_\ast=\sup\{t>0:\mu_D(t)>|D|(1-\frac{1}{9}A(D))\}$. By the proof of \cite[Theorem 1.3]{Brasco2019a}, one has
\begin{align}\label{eq:alphatorsion}
    T_\alpha(B)-T_\alpha(D)\geq C_{n,\alpha}T_\alpha(B)^2 (t_\ast)^{\frac{4}{\alpha}}A(D)^{\frac{2}{\alpha}}.
\end{align}
Let $\mu_B(t)=|\{y\in B: u_B^\alpha(y)>t\}|$. Since $u_B^\alpha(x)=C_{n,\alpha}(r^2-|x|^2)^{\frac{\alpha}{2}}$ and $r=\omega_n^{-\frac{1}{n}}$, we have
\begin{align*}
    \mu_B(t)=(1-C_{n,\alpha}t^{\frac{2}{\alpha}})^{\frac{n}{2}}.
\end{align*}
Choose $t_0>0$ such that $\mu_B(2t_0)=1-\frac{1}{18}A(D)$, then
\begin{align}\label{eq:toestalpha}
    t_0=C_{n,\alpha}(1-(1-\frac{1}{18}A(D))^{\frac{2}{n}})^{\frac{\alpha}{2}}
    \geq C_{n,\alpha}A(D)^{\frac{\alpha}{2}}.
\end{align}
If $t_\ast<t_0$, then $\mu_D(t)\leq 1-\frac{1}{9}A(D)$ for all $t\geq t_0$ by definition. Since $\mu_B(t)\geq 1-\frac{1}{18}A(D)$ for $t\leq 2t_0$, we get $\mu_B(t)-\mu_D(t)\geq \frac{1}{18}A(D)$ for $t\in[t_0,2t_0]$. By the layer cake representation and \eqref{eq:t0est}, we have
\begin{align*}
    T_\alpha(B)-T_\alpha(D)
    &= \int_0^\infty (\mu_B(t)-\mu_D(t))\, dt \\
    &\geq \frac{1}{18}t_0 A(D)\\
    &\geq C_{n,\alpha} A(D)^{1+\frac{\alpha}{2}}.
\end{align*}
If $t_\ast\geq t_0$, then by \eqref{eq:alphatorsion} and \eqref{eq:toestalpha} we have
\begin{align*}
    T_\alpha(B)-T_\alpha(D)
    \geq 
    C_{n,\alpha}T_\alpha(B)^2 A(D)^{2+\frac{2}{\alpha}}, 
\end{align*}
which completes the proof.
\qed

\section{Related open problems}\label{sec:prbls}

In this section, we discuss some open problems regarding the inequalities \eqref{eq:stableDist}, \eqref{eq:stableELI}, \eqref{eq:stable-pmmt}, and \eqref{eq:Talenti}. 

\subsection{Brownian motion}

It is open to find quantitative improvements of \eqref{eq:stableDist} and \eqref{eq:stable-pmmt} even for Brownian motion. In particular, it is unclear what is the right statement for stability of \eqref{eq:stableDist}. Having a small deficit of \eqref{eq:stableDist} at some $t$ is not enough to obtain proximity of the region to a ball, which implies that the deficit should be defined in a strong sense. 

As we discussed in Remark \ref{rmk:ellipse}, it is expected that the sharp exponent of \eqref{eq:quantLpelt} is 2 for $1<p\leq \infty$. For $p=1$, the sharp result was derived in \cite{Brasco2015a}. It is, however, not obvious how to apply the method of \cite{Brasco2015a} to the case $1<p\leq\infty$ because the proof strongly replies on the variational formula \eqref{eq:TD-varfor}, whereas the $L^p$ norm of the expected lifetime for $1<p\leq \infty$ does not have such formula.  

In Theorem \ref{mainthm1}, our quantitative result of \eqref{eq:stableELI} for $\alpha=2$ depends on $t_\ast$. It is unclear whether this dependence is necessary. Removing $t_\ast$ in \eqref{eq:main1} is an interesting open problem.

It was shown in \cite{Brasco2015a} that the sharp exponent of $A(D)$ in \eqref{eq:quantSV} is 2. Since the proof is based on the selection principle of \cite{Cicalese2012a}, the constant is not explicit. The best-known exponent with an explicit constant is 3. It is still open to prove a sharp quantitative result of \eqref{eq:quantSV} with a computable dimensional constant.

\subsection{Symmetric stable processes}

As mentioned above, it is open to extend Theorem \ref{mainthm1} and Theorem \ref{mainthm2} to the case $0<\alpha<2$. At this moment, a fractional analogue of the inequality \eqref{eq:Talenti} for $0<\alpha<2$ and $1<p\leq\infty$ is not known. Our approach of Theorem \ref{mainthm1} may not work for this case since it is not obvious how to apply the coarea formula in the fractional setting. A standard way of avoiding this difficulty is to consider the extension of Caffarelli--Silvestre \cite{Caffarelli2007a}. Fusco, Millot, and Morini \cite{Fusco2011a} considered the rearrangement inequality for the extension to show the quantitative isoperimetric inequality for the fractional perimeter. Recently, Brasco, Cinti, and Vita \cite{Brasco2019a} proved stability of the fractional Faber--Krahn inequality using a similar argument. As a corollary, they also showed stability of the fractional Saint-Venant inequality.  

The fractional Laplacian of order $\frac{\alpha}{2}$ is given by
\begin{equation*}
    (-\Delta)^{\frac{\alpha}{2}}f(x)
    =\sfA_{n,\alpha}\int_{\R^{n}}\frac{f(x)-f(y)}{|x-y|^{n+\alpha}}\,dy
\end{equation*}
where 
\begin{equation}\label{eq:A_nalpha}
    \sfA_{n,\alpha}
    =\frac
    {2^{\alpha}\Gamma\big(\frac{n+\alpha}{2}\big)}
    {\pi^{\frac{n}{2}}|\Gamma\big(-\frac{\alpha}{2}\big)|}.
\end{equation}
The space $\wt{W}^{\alpha,p}_{0}(D)$ is  the closure of $C_{0}^{\infty}(D)$ with respect to the norm $u\mapsto [u]_{\alpha,p}+\|u\|_{L^{p}(D)}$ where 
\begin{equation*}
    [u]_{\alpha,p}
    = \left(\int_{\R^{n}}\int_{\R^{n}}\frac{|u(x)-u(y)|^{p}}{|x-y|^{n+\alpha p/2}}\,dxdy\right)^{\frac{1}{p}}.
\end{equation*}    
The fractional torsional rigidity of order $\alpha$ is defined by $T_\alpha(D)=\|u_D^\alpha\|_1$ where $u_D^\alpha=\E^x[\tau_D^\alpha]$ is the expected lifetime of $X_t^\alpha$ in $D$. In this context, $u_D^\alpha$ is called the fractional $\alpha$-torsion function. We have the following variational representations
\begin{equation}\label{eq:Talpha-varfor}
    T_{\alpha}(D)
    = \max_{u\in \wt{W}^{\alpha,2}_{0}(D)\setminus\{0\}} 
        \left(
            2\|u\|_{L^1(D)}-\frac{\sfA_{n,\alpha}}{2}[u]_{\alpha,2}^{2}
        \right)
    =\max_{u\in \wt{W}^{\alpha,2}_{0}(D)\setminus\{0\}}
        \frac{2}{\sfA_{n,\alpha}}
        [u]_{\alpha,2}^{-2}\|u\|_{L^1(D)}^{2}
\end{equation}
where $\sfA_{n,\alpha}$ is given by \eqref{eq:A_nalpha}.
In particular, since $u^{\alpha}_{D}\in \wt{W}^{\alpha,2}_{0}(D)$ we have
\begin{equation}\label{eq:Talphafor}
    T_\alpha (D) 
    = \|u^{\alpha}_{D}\|_{L^1(D)}
    =\frac{\sfA_{n,\alpha}}{2}[u^{\alpha}_{D}]_{\alpha,2}^{2}
    =\frac{\sfA_{n,\alpha}}{2}\int_{\R^{n}}\int_{\R^{n}}\frac{|u^{\alpha}_{D}(x)-u^{\alpha}_{D}(y)|^{2}}{|x-y|^{n+\alpha}}\,dxdy.
\end{equation}
Consider a solution of the equation
\begin{align*}
    \begin{cases}
        \dive(z^{1-\alpha}\nabla U)=0, & (x,z)\in\R^{n+1}_+,\\
        U(x, 0)=u_D^\alpha(x), & x\in \R^n.
    \end{cases}
\end{align*}
Then we have
\begin{equation*}
    [u_D^\alpha]_{\alpha,2}^2=\gamma_{n,\alpha}\iint_{\R^{n+1}_+}z^{1-\alpha}|\nabla U|^2\,dxdz
    =\iint_{\R^{n+1}_+}z^{1-\alpha}|\nabla_x U|^2\,dxdz
    +\iint_{\R^{n+1}_+}z^{1-\alpha}|\partial_z U|^2\,dxdz
\end{equation*}
for some constant $\gamma_{n,\alpha}$. Let $U^\ast(x,z)=(U(\cdot,z))^\ast(x)$ be the symmetric decreasing rearrangement of $U$ with respect to $x$, then it was shown in \cite[Lemma 2.6]{Fusco2011a} that
\begin{align}\label{eq:polyaszegoforU}
    \iint_{\R^{n+1}_+}z^{1-\alpha}|\nabla_x U|^2\,dxdz
    \geq \iint_{\R^{n+1}_+}z^{1-\alpha}|\nabla_x U^\ast|^2\,dxdz
\end{align}
and
\begin{align*}
    \iint_{\R^{n+1}_+}z^{1-\alpha}|\partial_z U|^2\,dxdz
    \geq \iint_{\R^{n+1}_+}z^{1-\alpha}|\partial_z U^\ast|^2\,dxdz.
\end{align*}
In \cite{Brasco2019a}, the authors improved \eqref{eq:polyaszegoforU} quantitatively as in the local case, which leads to a quantitative fractional Saint-Venant inequality. 

To generalize Theorem \ref{mainthm1} and Theorem \ref{mainthm2} to  the $\alpha$-stable processes, one might need to apply this extension and a symmetrization argument at the level of the function $U$, not the seminorm $[u_D^\alpha]_{\alpha,2}$ as in \cite{Brasco2019a}. Then it is required to show that a quantitative improvement can be transferred as $z$ tends to 0. For $\alpha=1$, this approach was also used in \cite{Banuelos2004a, Banuelos2006a, Banuelos2006b} to study spectral gap estimates and properties of nodal domains.  Because of its connection to the Cauchy process and the Steklov problem, this special case may be more tractable with such an approach.  

\subsection{A fractional P\'olya--Szeg\"o inequality}
We discuss stability of a fractional P\'olya--Szeg\"o inequality, which has a close relation to the fractional Saint-Venant inequality. The fractional $\alpha$--perimeter of $D$ is defined by
\begin{equation*}
    P_{\alpha}(D)
    =\int_{D}\int_{\R^{n}\setminus D}\frac{1}{|x-y|^{n+\alpha/2}}dxdy
    =\frac{1}{2}[\chi_{D}]_{\alpha,1},
\end{equation*}
where $\chi_{D}$ is the characteristic function of $D$. Note that $P_{\alpha}(D)\geq C_{n,\alpha}|D|^{\frac{2n-\alpha}{2n}}$ by the fractional Sobolev embedding. The quantitative isoperimetric inequality for fractional perimeter \cite{Fusco2011a} states that for $n\geq 1$ and $\alpha\in(0,2)$, there exists a constant $\sfB_{n,\alpha}$ such that for all Borel set $D\subset \R^{n}$ with $0<|D|<\infty$, 
\begin{equation}\label{eq:frac_quant_isop}
    P_{\alpha}(D)
    \geq P_{\alpha}(D^{\ast})(1+\sfB_{n,\alpha}A(D)^{\frac{2}{\alpha}}).
\end{equation}
By the layer cake representation, we obtain a fractional version of the coarea formula \cite[Lemma 4.7]{Brasco2014b}. Indeed, if $u\in L^{1}(\R^{n})$ is a nonnegative function vanishing at $\infty$, then 
\begin{equation}\label{eq:frac_coarea}
    [u]_{\alpha,1}
    =2\int_{0}^{\infty} P_{\alpha}(\{x:u(x)>t\})dt.
\end{equation}
We have a fractional version of the P\'olya--Szeg\"o inequality with a remainder term.
\begin{proposition}\label{prop:frac_quant_PolyaSzego}
    Let $\alpha\in(0,2)$ and $D$ be a bounded domain in $\R^{n}$ with $A(D)>0$. If $u\in \wt{W}^{\alpha,1}_{0}(D)$, then there exists $t_{\ast}>0$ such that
    \begin{equation*}
        [u]_{\alpha,1}
        \geq [u^{\ast}]_{\alpha,1}
            +C_{n,\alpha}A(D)^{\frac{2}{\alpha}}\max\{t_{\ast}|D|^{\frac{2n-\alpha}{2n}},\|u\wedge t_{\ast}\|_{\frac{2n-\alpha}{2n}}\}.
    \end{equation*} 
\end{proposition}
\begin{proof}
Let $D_{t}=\{x:u(x)>t\}$ and $\mu(t)=|D_{t}|$. Using the coarea formula \eqref{eq:frac_coarea} and the quantitative isoperimetric inequality for fractional perimeter \eqref{eq:frac_quant_isop}, we have
\begin{align*}
    [u]_{\alpha,1}
    &= 2\int_{0}^{\infty} P_{\alpha}(D_t)dt \\
    &\geq  2\int_{0}^{\infty} P_{\alpha}(D_t^{\ast})dt
        +2\sfB_{n,\alpha}\int_{0}^{\infty}P_{\alpha}(D_t^{\ast})A(D_{t})^{\frac{2}{\alpha}} dt\\
    &\geq  [u^{\ast}]_{\alpha,1}
        +C_{n,\alpha}\int_{0}^{\infty}\mu(t)^{\frac{2n-\alpha}{2n}}A(D_{t})^{\frac{2}{\alpha}} dt
\end{align*}
for some constant $C_{n,\alpha}$. Let $t_{\ast}=\sup\{t>0: \mu(t)\geq |D|(1-\frac{1}{4}A(D))\}$. By Lemma \ref{lem:prop_asym} and \eqref{eq:levelset_est}, we have $A(D_{t})\geq \frac{1}{2}A(D)$ for $t<t_{\ast}$ and
\begin{equation*}
    [u]_{\alpha,1}
    \geq  [u^{\ast}]_{\alpha,1}
        +C_{n,\alpha}t_{\ast}|D|^{\frac{2n-\alpha}{2n}}A(D)^{\frac{2}{\alpha}}.
\end{equation*}
Using the inequality 
\begin{equation*}
    \left(\int_{0}^{\infty}f(x)dx\right)^{r}
    \geq \int_{0}^{\infty}rf(x)^{r}x^{r-1}dx
\end{equation*} 
for $r\geq 1$ and a nonnegative, non-increasing function $f$  on $(0,\infty)$ (see \cite[p.49]{Mazya2011a}), we get
\begin{equation*}
    \int_{0}^{t_{\ast}}\mu(t)^{\frac{1}{r}}dt
    \geq \left(\int_{0}^{t_{\ast}}rt^{r-1}\mu(t)dt\right)^{\frac{1}{r}}
    =\|u\wedge t_{\ast}\|_{r}
\end{equation*}
where $r=\frac{2n}{2n-\alpha}>1$, which implies
\begin{equation*}
    [u]_{\alpha,1}
    \geq  [u^{\ast}]_{\alpha,1}
        +C_{n,\alpha}\|u\wedge t_{\ast}\|_{\frac{2n}{2n-\alpha}}A(D)^{\frac{2}{\alpha}}.
\end{equation*}
\end{proof}

A natural question is a quantitative improvement of the inequality $[u]_{\alpha,2}\geq  [u^{\ast}]_{\alpha,2}$ in terms of $A(D)$. This open problem is interesting because it yields a quantitative Saint-Venant inequality. Suppose that we have $[u]_{\alpha,2}\geq  [u^{\ast}]_{\alpha,2}+\Phi(t_\ast, A(D))$ for some function $\Phi$. By \eqref{eq:Talpha-varfor}, \eqref{eq:Talphafor}, and the rearrangement inequality \cite{Frank2008b}, we get
\begin{align*}
    T_{\alpha}(D)
    &\leq \frac{2}{\sfA_{n,\alpha}}\frac{\|u^{\ast}\|_{1}^{2}}{[u^{\ast}]_{\alpha,2}^{2}+\Phi(t_{\ast},A(D))}\\
    &\leq T_{\alpha}(B)
        \left(1+\frac{\Phi(t_{\ast},A(D))}{[u^{\ast}]_{\alpha,2}^{2}}\right)^{-1}
\end{align*}
where $u=u_{D}^{\alpha}$ is the $\alpha$--torsion function and $B$ is a ball with $|D|=|B|$. Using the fact that $[u^{\ast}]_{\alpha,2}^{2}\leq [u]_{\alpha,2}^{2}$, we get
\begin{equation*}
    T_{\alpha}(B)-T_{\alpha}(D)
    \geq \Phi(t_{\ast},A(D)).
\end{equation*}
Under mild assumption on $\Phi$, $t_\ast$ can be removed as in Theorem \ref{mainthm2}. 

\subsection*{Acknowledgement} 
I would like to thank my academic advisor, Prof. Rodrigo Ba\~nuelos, for suggesting the problems investigated in this paper and his invaluable help and encouragement while writing the paper.

\providecommand{\bysame}{\leavevmode\hbox to3em{\hrulefill}\thinspace}
\providecommand{\MR}{\relax\ifhmode\unskip\space\fi MR }
\providecommand{\MRhref}[2]{%
  \href{http://www.ams.org/mathscinet-getitem?mr=#1}{#2}
}
\providecommand{\href}[2]{#2}

\end{document}